\documentclass{article}
\usepackage{amsthm,amsfonts, amsbsy, amssymb,amsmath,graphicx}
\usepackage{graphics}
\usepackage[english]{babel}
\newtheorem{thm}{Theorem}

\newtheorem{crl}{Corollary}

\newtheorem{st}{Statement}
\newtheorem{cj}{Conjecture}

\newcounter{tdfn}
\newenvironment{dfn}
{\vspace{0.15cm}{\bf Definition \arabic{tdfn}.}} {\par
\addtocounter{tdfn}{1}}

\newcounter{trk}
\newenvironment{rk}
{\vspace{0.15cm}{\bf Remark \arabic{trk}.}} {\par
\addtocounter{trk}{1}}

 \def\Z{{\mathbb Z}}
 \def\0{{\mathbbf 0}}
 \def\1{{\mathbbf 1}}

 \def\G{{\cal G}}

\newcommand{\skcrro}{\raisebox{-0.25\height}{\includegraphics[width=0.5cm]{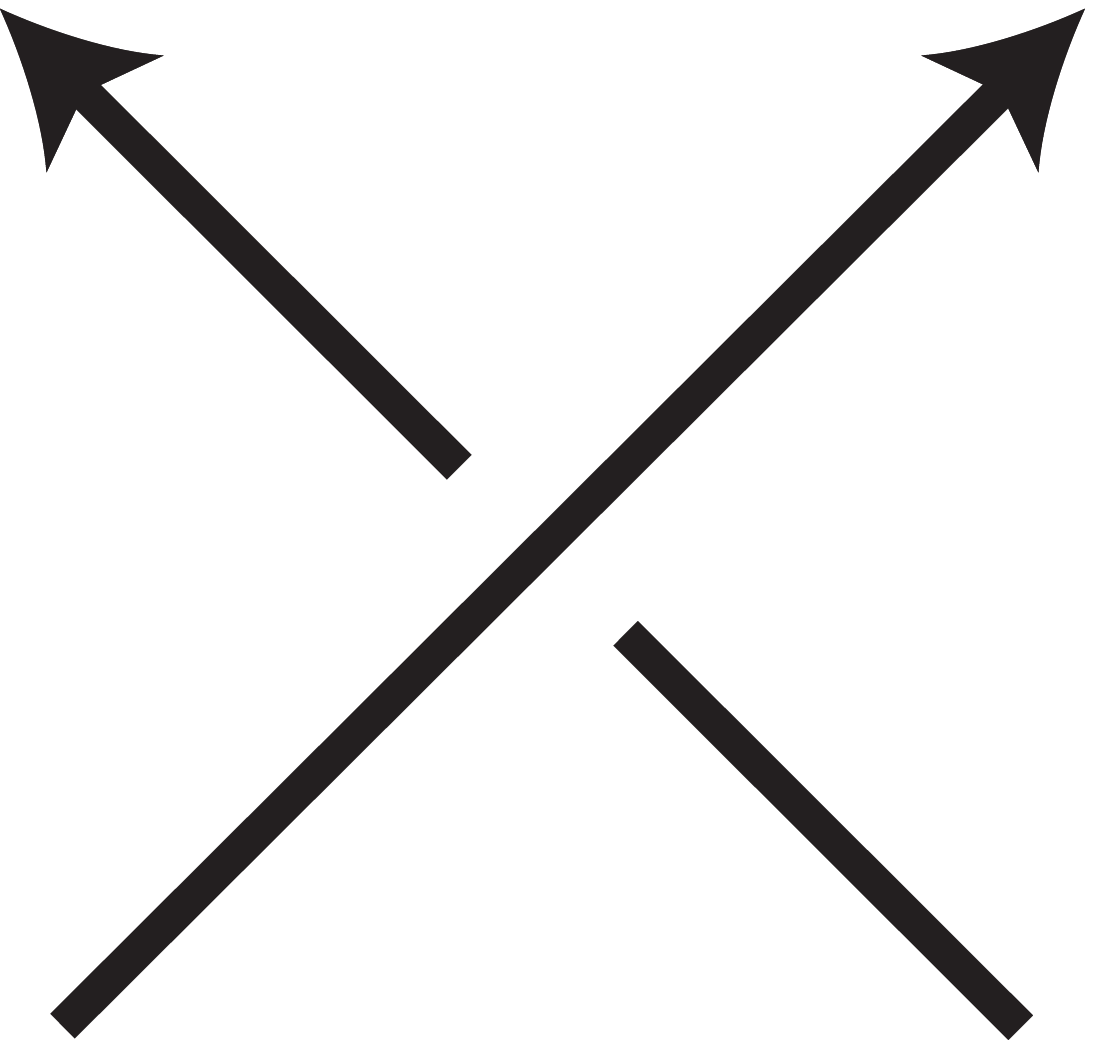}}}
\newcommand{\skcrlo}{\raisebox{-0.25\height}{\includegraphics[width=0.5cm]{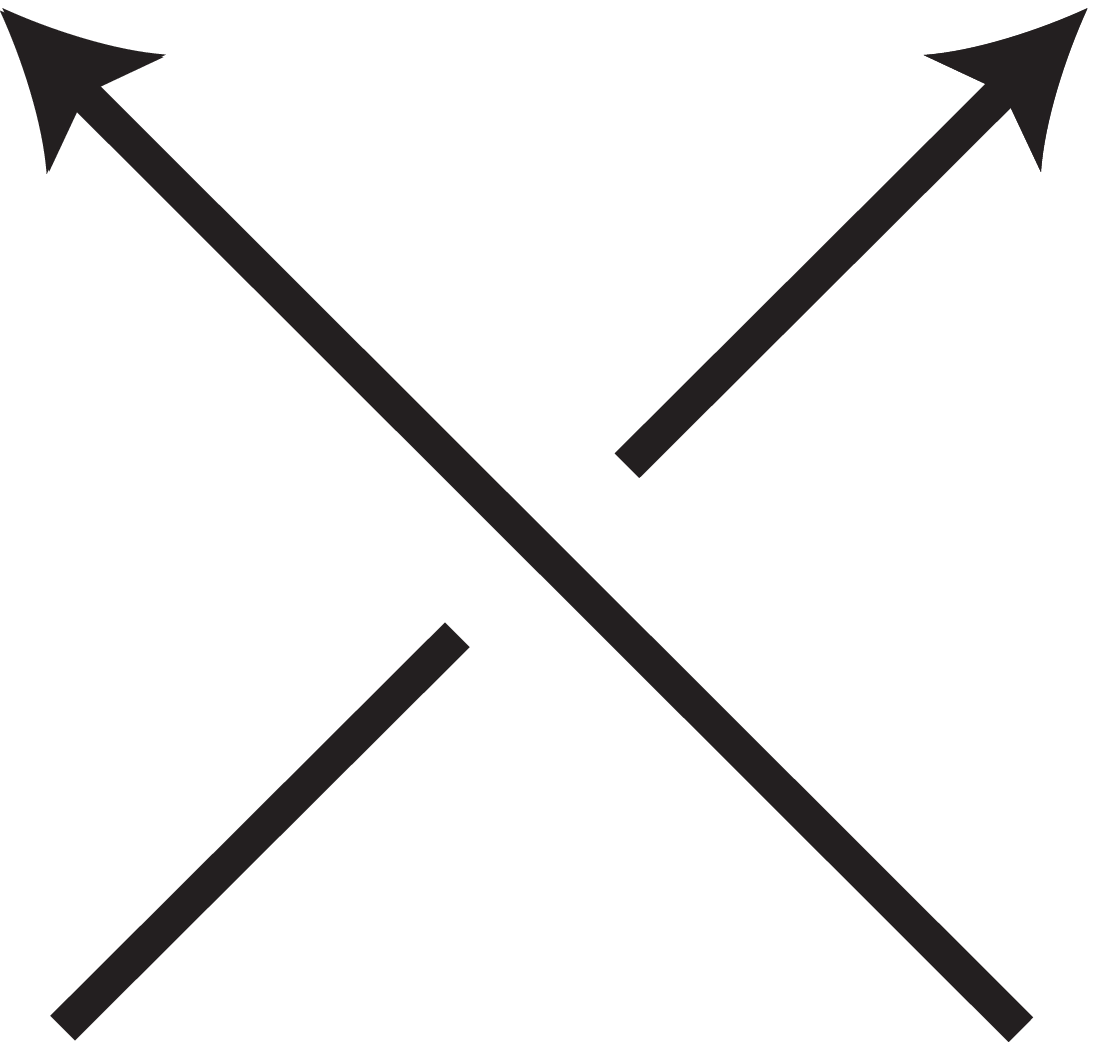}}}

\title{Parity and Projection from Virtual Knots to Classical Knots}

\author{Vassily Olegovich Manturov
\footnote{Peoples' Friendship University of Russia, Moscow 117198,
Ordjonikidze St., 3} \footnote{{\tt {vomanturov at yandex.ru}}}
\footnote{Partially supported by grants of the Russian Government
11.G34.31.0053, RF President NSh – 1410.2012.1, Ministry of
Education and Science of the Russian Federation 14.740.11.0794.} }

\begin{document}

\maketitle

\begin{abstract}
We construct various functorial maps (projections) from virtual
knots to classical knots. These maps are defined on diagrams of
virtual knots; in terms of Gauss diagram each of them can be
represented as a deletion of some chords. The construction relies
upon the notion of parity. As corollaries, we prove that the minimal
classical crossing number for classical knots.

Such projections can be useful for lifting invariants from classical
knots to virtual knots.

Different maps satisfy different properties.
\end{abstract}



\section{Introduction. Basic Notions}

Classical knot theory studies the embeddings of a circle (several
circles) to the plane up to isotopy in three-space. Virtual knot
theory studies the  embeddings of curves in thickened oriented
surfaces of arbitrary genus, up to the addition and removal of empty
handles from the surface. Virtual knots have a special diagrammatic
theory, described below, that makes handling them very similar to
the handling of classical knot diagrams. Many structures in
classical knot theory generalize to the virtual domain directly,
however, many other required more techniques \cite{MaIl};
nevertheless, many other structures (like Heegaard-Floer homology)
have not been generealized to virtual knots so far; the existence of
a well-defined projection from virtual knot theory to classical knot
theory may help solving such problems.

In the diagrammatic theory of virtual knots one adds a {\em virtual
crossing} (see Figure~\ref{Figure 1}) that is neither an
overcrossing nor an undercrossing.  A virtual crossing is
represented by two crossing segments with a small circle placed
around the crossing point. Figures \ref{Figure 1} and \ref{Figure 4}
are borrowed from \cite{KM2}.

Note that a classical knot vertex is a $4$-valent graphical node
embedded in the plane with extra structure. The extra structure
includes the diagrammatic choice of crossing (indicated by a broken
segment) and a specific choice of cyclic order (counterclockwise
when embedded in the plane) at the vertex. By a {\em framing} of a
four-valent graph we mean a splitting of the four emanating
(half)edges into two pairs of opposite (half)edges. The
counterclockwise cyclic order includes more information than just a
framing.
 A virtual knot is completely
specified by its $4$-valent nodes with their cyclic structure if the
edges incident to the nodes are labeled so that they can be
connected by arcs to form the corresponding graph.

Throughout the paper, all knots are assumed oriented. The results of
this paper are about virtual knots, as stated; nevertheless, after a
small effort they can be upgraded for the case of virtual links.

A {\em virtual diagram} is an immersion of a collection of circles
into the plane such that some crossings are structured as classical
crossings and some are simply labeled as virtual crossings and
indicated by a small circle drawn around the crossing. We regard the
resulting diagram as a possible non-planar graph whose only nodes
are the classical crossings, with their cyclic structure. Any
immersion of such a graph, preserving the cyclic structure at the
nodes, will represent the {\it same} virtual knot or link. From
this, we use the {\it detour move} (see below) for arcs with
consecutive virtual crossings, so that this equivalence is
satisfied. For the projection of the unknot (unlink) without
classical crossings we shall also admit a circle instead of graph;
thus, we category of graphs includes the circle.

Immersion of each particular circle from the collection gives rise
to a {\em component} of a virtual link diagram; virtual link
diagrams with one component are {\em virtual knot diagrams}; we
shall deal mostly with virtual knots and their diagrams, unless
specified otherwise; (virtual) {\em knots} are one-component
(virtual) links.

Moves on virtual diagrams generalize the Reidemeister moves
(together with obvious planar isotopy) for classical pieces of knot
and link diagrams (Figure~\ref{Figure 1}). One can summarize the
moves on virtual diagrams by saying that the classical crossings
interact with one another according to the usual Reidemeister moves
while virtual crossings are artifacts of the attempt to draw the
virtual structure in the plane. A segment of diagram consisting of a
sequence of consecutive virtual crossings can be excised and a new
connection made between the resulting free ends. If the new
connecting segment intersects the remaining diagram (transversally)
then each new intersection is taken to be virtual. Such an excision
and reconnection is called a {\it detour move}. Adding the global
detour move to the Reidemeister moves completes the description of
moves on virtual diagrams. In Figure~\ref{Figure 1} we illustrate a
set of local moves involving virtual crossings. The global detour
move is a consequence of  moves (B) and (C) in Figure~\ref{Figure
1}. The detour move is illustrated in Figure~\ref{Figure 2}. Virtual
knot and link diagrams that can be connected by a finite sequence of
these moves are said to be {\it equivalent} or {\it virtually
isotopic}. A virtual knot is an equivalence class of virtual
diagrams under these moves.

\begin{figure}[htb]
     \begin{center}
     \begin{tabular}{c}
     \includegraphics[width=10cm]{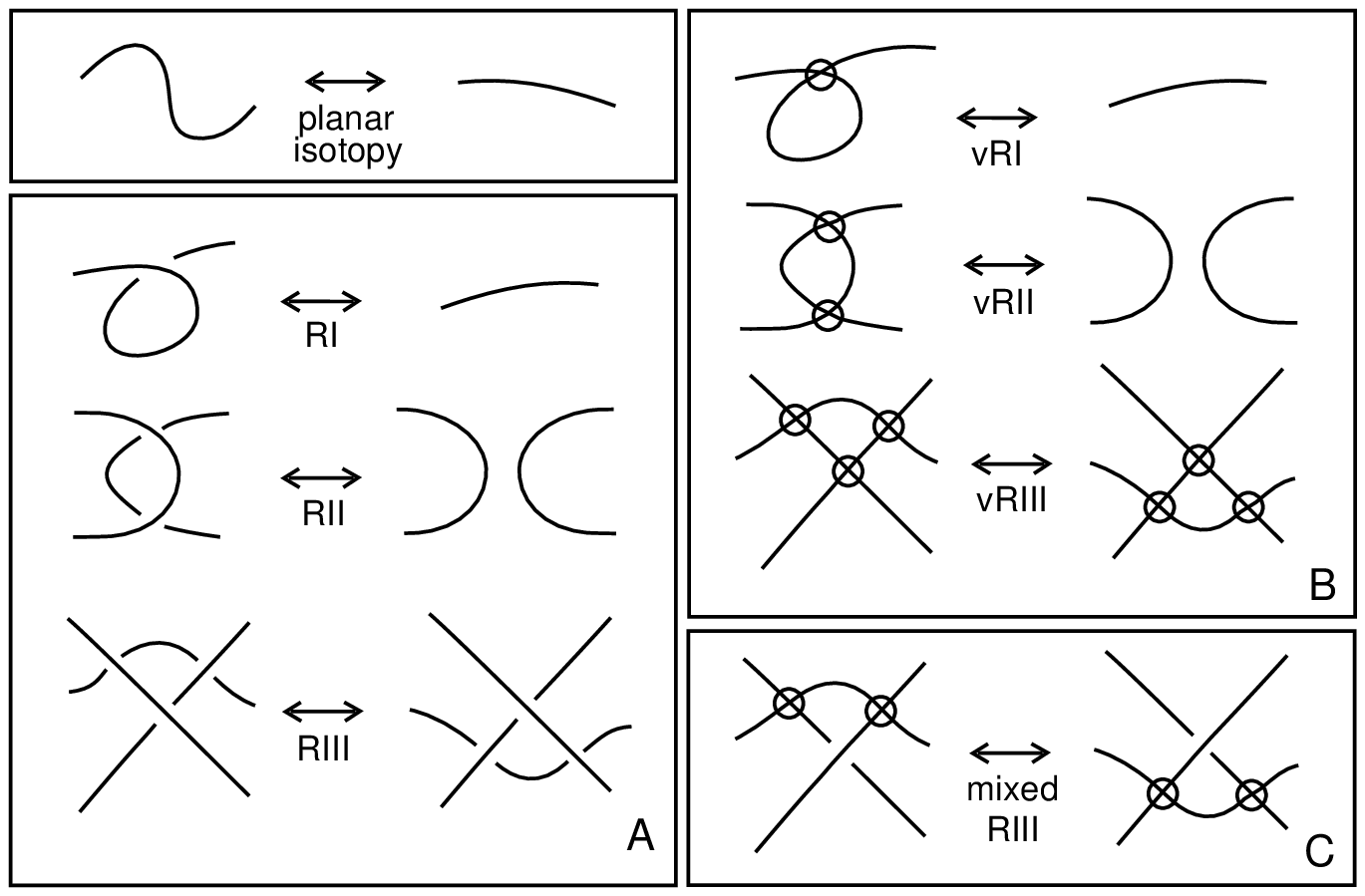}
     \end{tabular}
     \caption{\bf Moves}
     \label{Figure 1}
\end{center}
\end{figure}

\begin{figure}[htb]
     \begin{center}
     \begin{tabular}{c}
     \includegraphics[width=10cm]{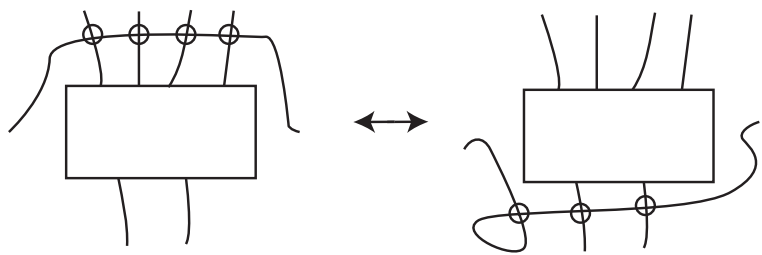}
     \end{tabular}
     \caption{\bf Detour Move}
     \label{Figure 2}
\end{center}
\end{figure}

Another way to understand virtual diagrams is to regard them as
representatives for oriented Gauss diagrams \cite{GPV}. The Gauss
diagram encodes the information about classical crossings of a knot
diagram and the way they are connected. However, not every Gauss
diagram has a planar realization.

An attempt to draw the corresponding diagram on the plane leads to
the production of the virtual crossings. Gauss  diagrams are most
convenient for knots, where there is one cycle in the code and one
circle in the Gauss diagram. One can work with Gauss diagrams for
links with a little bit more care, but we will not touch on this
subject.

The detour move makes the particular choice of virtual crossings
irrelevant.

{\it Virtual isotopy is the same as the equivalence relation
generated on the collection of oriented Gauss diagrams by abstract
Reidemeister moves on these codes.}

The paper is organized as follows. In the end of the introduction,
we present all necessary constructions of Gauss diagrams, band
presentation, and parity.

In Section 2, we formulate the main theorem (about projection) and
prove it modulo some important auxiliary theorems, one of them due
to I.M.Nikonov. We also prove two corollaries from the main theorem.

Section 3 is devoted to the proof of basic lemmas.

In Section 4, we introduce {\em parity groups} and discuss other
possibilities of constructing projection maps from virtual knots to
classical knots.

The paper is concluded by Section 5, where we discuss some obstacles
which do not allow us to define the projection uniquely on the
diagrammatic level.

\subsection{Acknowledgements}

I am grateful to L.H.Kauffman, I.M.Nikonov, V.V.Chernov, D.P.Ilyutko
for various fruitful discussions.

\subsection{Gauss diagrams}

\begin{dfn}
A {\em Gauss daigram} is a finite trivalent graph which consists of
just an oriented cycle passing through all vertices (this cycle is
called the {\em core} of the Gauss diagrams) and a collection of
oriented edges ({\em chords}) connecting crosssings to each other.
Besides the orientation, every chord is endowed with a sign.

Besides that we consider the {\em empty Gauss diagram} which is not
a graph, but an oriented circle; this empty Gauss diagram
corresponds to the unknot diagram without crossings.
\end{dfn}

Given a one-component virtual diagram $D$. Let us associate with it
the following Gauss diagram $\G(D)$. Let us represent the framed
four-valent graph $\Gamma$ of the diagram $D$ as the result of
pasting of a closed curve at some points (corresponding to classical
crossings) in such a way that the two parts of the neighbourhood of
a pasted point are mapped to {\em opposite} edges at the crossing.

Thus, we have a map $f:S^{1}\to \Gamma$. For the {\em core circle}
of the chord diagram we take
 $S^{1}$, vertices of the chord diagrams are preimages of vertices of  $\Gamma$,
 and chords connect those pairs of vertices having the same image.
The orientation of the circle corresponds to the orientation of the
knot. Besides, the chord
 is directed from the preimage of the overcrossing arc to the preimage of
an undercrossing arc; the sign of the chord is positive for
crossings of type $\skcrro$ and negative for crossings of type
$\skcrlo$.

We say that a Gauss diagram is {\em classical} if it can be
represented by a classical diagram (embedding of a four-valent graph
without virtual crossings). In Fig. 3, Reidemeister moves for Gauss
diagrams are drawn without indication of signs and arrows. For
Reidemeister-1 (the upper picture) move an addition/removal of a
solitary chord of any sign and with any arrow direction is possible.
For Reidemeister-2 move (two middle pictures), the chords $a$ and
$b$ should have the same orientation, but different signs.

The articulation for the third Reidemeister move (lowest picture) is
left for the reader as an exercise.

Note that two the Gauss diagram does not feel the detour move: if
two diagrams $K,K'$ are virtually isotopic, then $\G(K)=\G(K')$.

We say that a virtual knot diagram $K_{1}$ is {\em smaller} than the
diagram $K_{2}$, if the Gauss diagram of $K_{1}$ is obtained from
that of $K_{2}$ by a deletion of some chords.

For this, we take the notation $K_{1}<K_{2}$.

 As usual, we make no distinction
between virtually isotopic diagrams.

 This introduces a partial ordering on the set of
virtual knot diagrams. The unknot diagram without classical
crossings is smaller than any diagram with classical crossings.

\begin{figure}
     \begin{center}
     \begin{tabular}{c}
     \includegraphics[width=6cm]{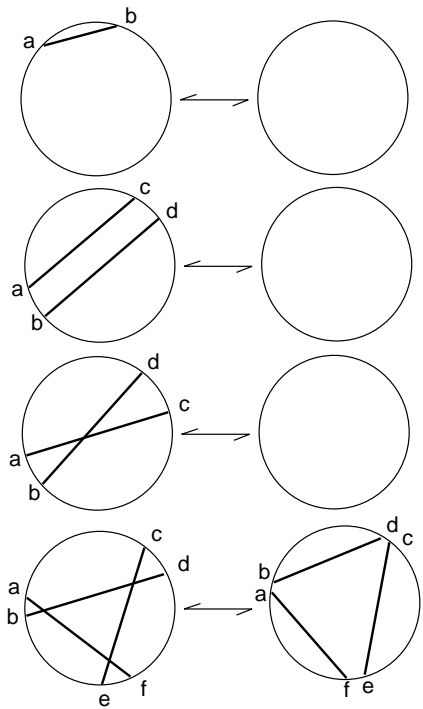}
     \end{tabular}
     \caption{\bf Reidemeister Moves on Chord Diagrams}
     \label{fig3}
\end{center}
\end{figure}

Having a Gauss diagram, one gets a collection of classical crossings
with an indication how they are connected to each other. So, a Gauss
diagram leads to a {\em virtual equivalence classes} of virtual knot
diagrams (note that Gauss diagram carries no information about
virtual crossings, so, virtually equivalent diagrams diagrams lead
to the same Gauss diagram).

By a {\em bridge} \cite{CSV} of a Gauss diagram we mean an arc of
the core circle between two adjacent arrowtails (for the edge
orientation for the chords of the chord diagram) containing
arrowheads only (possibly, none of them). In the corresponding
planar diagram, a {\em bridge} is a branch of the knot diagram from
an undercrossing to the next undercrossing containing overcrossings
and virtual crossings only. Thus, every virtual knot diagrams
naturally splits into bridges, see Fig. \ref{trefbrid}.

\begin{figure}
     \begin{center}
     \begin{tabular}{c}
     \includegraphics[width=6cm]{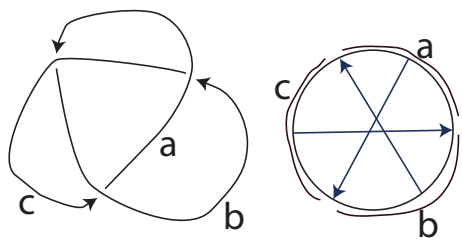}
     \end{tabular}
     \caption{\bf The Trefoil Knot and its Bridges}
     \label{trefbrid}
\end{center}
\end{figure}

 The {\em bridge number}
of a virtual knot diagram is the minimal number of its bridges.
Since the bridge number is defined in terms of Gauss diagram, it
does not change under detour moves.

With this, one can define the {\em minimal crossing number} and the
{\em bridge number} for virtual knots to be the minimum of crossing
numbers (resp., bridge numbers) over all virtual knot diagrams
representing the given knot. When we restrict to classical knots,
there we also have the definition when the minima are taken only
over classical diagrams.

So, for crossing number and bridge number for classical knots, we
have two definitions, the {\em classical one} and the {\em virtual
one}. As we shall see in the present paper (Corollaries \ref{crl1},
\ref{bridge}), these two definitions coincide, moreover, any virtual
diagram of a classical knot where the minimal classical crossing
number (resp., minimal bridge number) is obtained, is in fact,
virtually equivalent to a classical one.

\subsection{Band Presentation of Virtual Knots}

Note that knots in a thickened surface $S_{g}\times I$ are encoded
by regular projections on $S_{g}$ with over and undercrossings and
no virtual crossings. These diagrams are subject to classical
Reidemeister moves which look locally precisely as in the classical
case. No detour moves are needed since we have no virtual crossings
for such diagrams.

Let ${\cal K}$ be a (class of a) virtual knot, given by some virtual
diagram $K$. Let us describe the {\em band presentation} of this
knot as a knot in a thickened surface (following N.Kamada and
N.Kamada \cite{KK}).

We shall construct a surface $S(K)$ corresponding to the diagram
$K$, as follows. First, we construct a surface with boundary
corresponding to $K$.

With every classical crossing, we associate a ``cross'' (upper
picture in Fig. \ref{cr}), and with every virtual crossing, we
associate a pair of ``skew'' bands (lower part of Fig. \ref{cr}).

Connecting these crosses and bands by non-intersecting and
non-twisted bands going along the edges of the diagram, we get an
oriented $2$-manifold with boundary, to be denoted by $S'(K)$ (the
orientation is taken from the plane), see Fig. \ref{Figure 4}.

\begin{figure}
\begin{center}
\begin{picture}(100,160)
\thicklines \put(5,95){\line(1,1){50}} \put(55,95){\line(-1,1){50}}
\put(50,120){$\longrightarrow$} \put(50,40){$\longrightarrow$}
\put(5,5){\line(1,1){50}} \put(55,5){\line(-1,1){50}}
\put(30,30){\circle{5}} \thinlines \put(65,100){\line(1,1){20}}
\put(85,120){\line(-1,1){20}} \put(70,95){\line(1,1){20}}
\put(90,115){\line(1,-1){20}} \put(115,100){\line(-1,1){20}}
\put(95,120){\line(1,1){20}} \put(110,145){\line(-1,-1){20}}
\put(90,125){\line(-1,1){20}} \put(65,10){\line(1,1){50}}
\put(70,5){\line(1,1){50}} \put(115,5){\line(-1,1){20}}
\put(120,10){\line(-1,1){20}} \put(65,55){\line(1,-1){20}}
\put(70,60){\line(1,-1){20}}
\end{picture}
\end{center}
\vspace{-0.5cm} \caption{Local Structure of $M'$} \label{cr}
\end{figure}
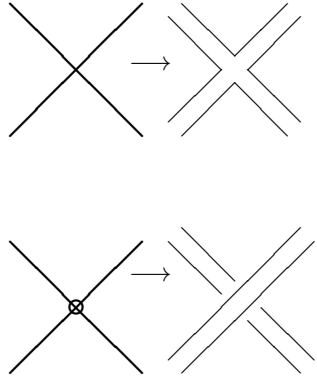

The diagram  $K$ can be drawn on the surface  $S'(K)$ in a natural
way so that the arcs of the diagram (which may pass through virtual
crossings) are located in such a way that the arcs are go along the
middle lines of the band, and classical (flat) crossings correspond
to intersection of middle lines inside crossings. Thus we get  curve
$\delta\subset S'(K)$ (for a link we would get a set of curves).
Pasting the boundary components of the manifold $S'(K)$ by discs, we
get an oriented manifold $S=S(K)$ without boundary with a curve
$\delta$ in it; we call the surface $S(K)$  {\em the underlying
surface for the diagram $K$}. We call the genus of this surface the
{\em underlying diagram genus} of the diagram $K$.

We call the connected components of the boundary of $S'(K)$ the {\em
pasted cycles} or the {\em rotating cycles}. Originally rotating
cycles are defined by using source-sink orientation of $K$, but in
this paper we regard them as the boundary of the oriented surface
$S'(D)$ since we handle diagrams which do or do not admit a
source-sink orientation. These pasted cycles treated as collections
of vertices, will be used in the sequel for constructing parity
groups.

By the {\em underlying genus} of a virtual knot we mean the minimum
of all underlying genera over all diagrams of this knot.

We say that a diagram $K$ is a {\em minimal genus diagram} if the
genus of the diagram coincides with the genus of the corresponding
knot.

As we shall see, some minimal characteristics of virtual knots can
be realized only on minimal genus diagrams.

The detour move does not change the band presentation of the knot at
all. As for Reidemeister move, the first and the third moves do not
change the genus of the knot, whence the second
increasing/decreasing move may increase/decrease the genus of the
underlying surface (cause stabilization/destabilization).

To define handle stabilization, regard the knot or link as
represented by a diagram $D$ on a surface $S.$  If $C$ is an
embedded curve in $S$ that does not intersect the diagram $D$ and
cutting along $D$ does not disconnect the surface, then we cut along
$C$ and add two disks to fill in the boundary of the cut surface.
This is a handle destabilization move that reduces the genus of the
surface to a surface $S'$ containing a new diagram $D'.$ The pairs
$(S,D)$ and $(S',D')$ represent the same virtual knot or link. The
reverse operation that takes $(S',D')$  to  $(S,D)$ consists in
choosing two disks in $S'$ that are disjoint from $D'$, cutting them
out and joining their boundaries by a tube (hence the term handle
addition for this direction of stabilization).

\begin{figure}
     \begin{center}
     \begin{tabular}{c}
     \includegraphics[width=10cm]{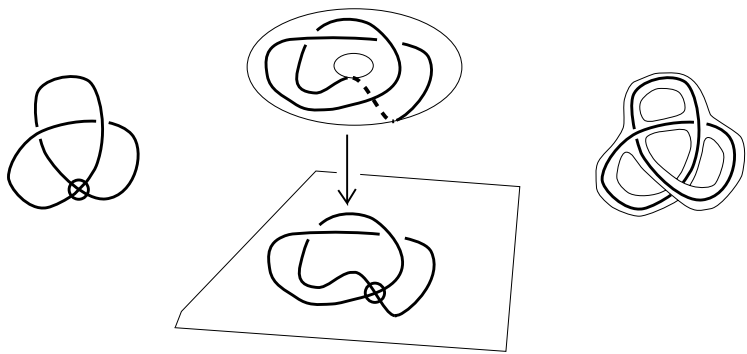}
     \end{tabular}
     \caption{ Surfaces and Virtual Knots}
     \label{Figure 4}
\end{center}
\end{figure}

We say that two such surface embeddings are {\em stably equivalent}
if one can be obtained from another by isotopy in the thickened
surfaces, homeomorphisms of the surfaces and handle stabilization.

\begin{thm}[\cite{KK,KaV}]
The above description of a band representation leads to a bijection
between virtual knots and stably equivalent classes of embeddings of
circles in thickened surfaces.
\end{thm}

So, we shall deal with the following two equivalences: the usual one
(with (de)sta\-bi\-li\-sa\-tion) and the equivalence without
(de)stabilisation which preserves the genus of the underlying
surface.

The  Kuperberg Theorem says that virtual knots can be studied by
using their minimal representatives. More precisely, we have

\begin{thm}[Kuperberg's Theorem,\cite{Kup}]
A minimal genus diagram of a virtual knot ${\cal K}$ is unique up to
isotopy; in other words, if two diagrams $K_{1},K_{2}$ are of the
minimal genus then there is a sequence of Reidemeister moves from
$K_{1}$ to $K_{2}$ such that all intermediate diagrams between
$K_{1}$ and $K_{2}$ are of the same genus.
\end{thm}

\subsection{Parity}

\begin{dfn}
Let ${\cal L}$ be a knot theory, i.e., a theory whose objects are
encoded by diagrams (four-valent framed graphs, possibly, with
further decorations) modulo the three Reidemeister moves (and the
detour move) applied to crossings. For every Reidemeister move
transforming a diagram $K$ to a diagram $K_{1}$ there are
corresponding crossings: those crossings outside the domain of the
Reidemeister move for $K$ are in one-to-one correspondence with
those crossings outside the domain of the Reidemeister move for
$K_{1}$. Besides, for every third Reidemeister move $K\to K_{1}$
there is a natural correspondence between crossings of $K$ taking
part in this move and the resulting crossings of $K_{1}$. By a {\em
parity} for the knot theory ${\cal L}$ we mean a rule for
associating $0$ or $1$ with every (classical) crossing of any
diagram $K$ from the theory ${\cal L}$ in a way such that:

\begin{enumerate}

\item For every Reidemeister moves $K\to K_{1}$ the corresponding crossings have the same parity;

\item For each of the three Reidemeister moves the sum of parities of crossings taking part
in this move is zero modulo two.

\end{enumerate}

\end{dfn}

\begin{dfn} Now, a {\em parity in a weak sense} is defined in the same way as parity
but with the second condition relaxed for the case of the third
Reidemeister move. We allow three crossings taking part in the third
Reidemeister move to be all odd (so for the third Reidemeister move
the only forbidden case is when exactly one of three crossings is
odd).
\end{dfn}

We shall deal with parities for {\em virtual knots} or for {\em
knots in a given thickened surface}. In the latter case diagrams are
drawn on a $2$-surface and Reidemeister moves are applied to these
diagrams; no ``stabilizing'' Reidemeister moves changing the genus
of the surface are allowed.

We say that two chords of a Gauss diagram $a,b$ are {\em linked} if
two ends of one chord $a$ belong to different connected components
of the complement to the endpoints of $b$ in the core circle of the
Gauss diagram (it is assumed that no chord is linked with itself).
We say that a chord of a Gauss diagram is {\em even} (with respect
to the {\em Gaussian parity}) if it is linked with evenly many
chords; otherwise we say that this chord is {\em odd} (with respect
to the {\em Gaussian parity}). We shall say that a classical
crossing of a virtual knot diagram is even whenever the
corresponding chord is even. One can easily check the parity axioms
for the Gaussian parity.

For every parity $p$ for virtual knots (or knots in a specific
thickened surface), consider a mapping $pr_{p}:{\cal G}\to {\cal G}$
from the set of Gauss diagrams ${\cal G}$ to itself, defined as
follows. For every virtual knot diagram $K$ represented by a Gauss
diagram ${\cal G}(K)$ we take $pr_{p}(K)$ to be the virtual knot
diagram represented by the Gauss diagram obtained from ${\cal G}(K)$
by deleting odd chords with respect to $p$. At the level of planar
diagrams this means that we replace odd crossings by virtual
crossings.

The following theorem follows from definitions,
see,e.g.,\cite{Sbornik1}.
\begin{thm}
The mapping $pr_{p}$ is well defined, i.e., if $K$ and $K'$ are
equivalent, then so are $pr_{p}(K)$ and $pr_{p}(K')$.

The same is true for every parity in a weak sense as discussed
above. \label{gsthm}
\end{thm}

Thus, for the Gaussian parity $g$ one has a well-defined projection
$pr_{g}$. Note that if $K$ is a virtual knot diagram, then
$pr_{g}(K)$ might have odd chords: indeed, some crossings which were
even in $K$ may become odd in $pr_{g}(K)$.

However, this map $pr_{g}$ may take diagrams from one theory to
another; for example, if we consider equivalent knots lying in a
given thickened surface, their images should not necessarily be
realised in the same surface; they will just be equivalent virtual
knots. For virtual knots, this is just a map from virtual knots to
virtual knots.

Note that $pr_{g}$ is not an idempotent map. For example, if we take
the Gauss diagram with four chords $a,b,c,d$ where $a$ is linked
with $b,c$, the chord $b$ is linked with $a,d$, the chord $c$ is
linked with $a$, and the chord $d$ is linked with $b$, then after
applying $pr_{g}$, we shall get a diagram with two chords $a,b$, and
they will both become odd, see Fig. \ref{notidemp}.

\begin{figure}
     \begin{center}
     \begin{tabular}{c}
     \includegraphics[width=7cm]{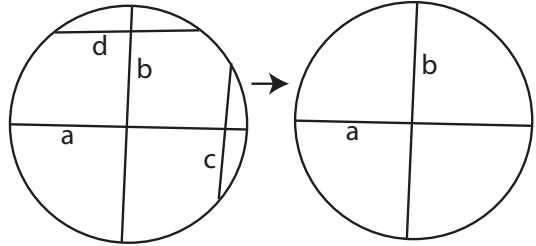}
     \end{tabular}
     \caption{The parity projection is not idempotent}
     \label{notidemp}
\end{center}
\end{figure}

Now, let $S_{g}$ be a surface of genus $g$. Fix a cohomology class
$\alpha\in H^{1}(S_{g},\Z_{2})$. Let us consider those knots $K$ in
$S_{g}$ for which the total homology class of the knot $K$ in
$H_{1}(S_{g},\Z_{2})$ is trivial.

With every crossing $v$ of $K$ we associate the two {\em halves}
$h_{v,1},h_{v,2}$ (elements of the fundamental group
$\pi_{1}(S_{g},v)$) as follows. Let us smooth the diagram $K$ at $v$
according to the orientation of $K$. Thus, we get a two-component
oriented link. If $\alpha(h_{v,1})=\alpha(h_{v,2})=0$ we say that
the crossing $v$ is {\em even}; otherwise we say that it is {\em
odd}.

In \cite{IMN} it is proved that this leads to a well-defined parity
for knots in $S_{g}\times I$. Thus, every $\Z_{2}$-cohomology class
of the surface which evaluates trivially on the knot itself, gives
rise to a well-defined parity. We shall call it the {\em homological
parity}.

\section{Statements of Main Results}

For every Gauss diagram one can decree some chords (crossings) to be
{\em true classical} (in an ambiguous way, see discussion in the
last section) and remove the other ones, so that  the resulting
Gauss diagram classical, and this map will give rise to a
well-defined projection from virtual knots to classical knots. In
Fig. \ref{vkclasproj}, a virtual knot $A$ is drawn in the left part;
its band presentation belongs to the thickened torus (see upper part
of the right picture); there are four ``homologically non-trivial''
crossings disappear which leads to the diagram $D$ (virtually
isotopic to the one depicted in the lower picture of the right
half). This is the classical trefoil knot diagram.

The aim of the present article is the proof of the following

\begin{thm}
For every virtual diagram $K$ there exists a classical diagram
${\bar K}$, such that:

\begin{enumerate}

\item ${\bar K}<K$;

\item ${\bar K}=K$ if and only if $K$ is classical.

\item If $K_{1}$ and $K_{2}$ are equivalent virtual knots, then so do ${\bar K_{1}}$ and ${\bar K_{2}}$.

\item The map restricted to non-classical knots is a surjection onto
the set of all classical knots.

\end{enumerate}
\label{mainthm}
\end{thm}

The discrimination between ``true classical'' crossings and those
crossings which will become virtual is of the topological nature, as
we shall see in the proof of Theorem \ref{mainthm}.

As usual, we make no distinction between virtually isotopic
diagrams: a  virtual diagram is said to be {\em classical} if the
corresponding Gauss diagram represents a classical knot.

Thus, it makes sense to speak about a map from the set of virtual
knots to the set of classical knots. This map will be useful for
lifting invariants from virtual knots to classical knots.

We shall denote this map by $K\to f(K)$ where $K$ means the knot
type represented by $K$, and $f(K)$ means the resulting knot type of
the corresponding classical knots.

The only statement of the theorem which deals with diagrams of knots
which are not classical, is 4). Otherwise we could just project all
diagarams which do not represent classical knots to the unknot
diagram (without classical crossings), and the functorial map would
be rather trivial.

Nevertheless, as we shall see, one can construct various maps of
this sort. Different proofs of Theorem \ref{mainthm} can be used for
constructing various functorial maps and establishing properties of
knot invariants.

A desired projection would be one for which there is a well defined
mapping at the level of Gauss diagrams, and the projection is such
that if any two diagrams which are connected by a Reidemeister
moves, their images are connected by the same Reidemeister move or
by a detour move. Unfortunately, such projections seem not to exist
(see the discussion in the end of the paper); see also Nikonov's
Lemma (Theorem \ref{lmnik}).

For example, based on the notion of weak parity and parity groups,
we shall construct another projection satisfying the conditions of
Theorem \ref{mainthm}; the construction will not be in two turns as
in the case when Nikonov's lemma is applied; however, this map will
``save'' more classical crossings.

From Theorem \ref{mainthm} we have the following two corollaries
\begin{crl}
Let ${\cal K}$ be an isotopy class of a classical knot. Then the
minimal number of classical crossings for virtual diagrams of ${\cal
K}$ is realized on classical diagrams (and those obtained from them
by the detour move). For every non-classical diagram realizing a
knot from ${\cal K}$, the number of classical crossings is strictly
greater than the minimal number of classical crossings.

Moreover, minimal classical crossing number of a non-classical
virtual knot is realized only on minimal genus diagrams.
 \label{crl1}
\end{crl}

Indeed, the projection map from the main theorem decreases the
number of classical crossings, and preserves the knot type.

The observation that the following corollary is a consequence from
Theorem \ref{mainthm} is due to V.V.Chernov (Tchernov).

\begin{crl}
Let ${\cal K}$ be a classical knot class. Then the bridge number for
the class ${\cal K}$ can be realized on classical diagrams of $K$
only.

Moreover, minimal bridge number of a non-classical virtual knot is
realized on minimal genus diagrams (here we do not claim that it can
not be realized on non-classical daigrams).

\label{bridge}
\end{crl}

\begin{proof}
Indeed, it suffices to see that if $K'<K$ then $br(K')\le br(K)$:
when replacing a classical crossing with a virtual crossing, the
number of bridges cannot be increased; it can only decrease because
two bridges can join to form one bridge.\label{crl2}
\end{proof}

\begin{rk}
We do not claim that the diagram $K'$ representing the class $f(K)$
is unique. In fact, we shall construct many maps satisfying the
conditions of Theorem \ref{mainthm}. In the last section of the
present work we discuss the question, to which extent the diagram
$K'$ can be defined uniquely by the diagram $K$, see the discussion
in the last section of the paper.
\end{rk}

Theorem \ref{mainthm} allows one to lift invariants of classical
knots to virtual knots. The straighforward way to do it is to
compose the projection with the invariant in question. However,
there is another way of doing it where crossings which are not
classical, are not completely forgotten (made virtual) but are
treated in another way than just usual ``true classical'' crossings.
In similar cases when projection is well defined at the level of
diagrams, this was done in \cite{Sbornik1,Af} etc.: in these papers
 a distinction between even and odd crossings was taken into account
 to refine many known invariants (note that, according to the parity
 projection map, one can completely disregard odd crossings; on the
 other hand, they can be treated as classical crossings as they were
 from the very beginning).

The proof of Theorem \ref{mainthm} is proved in two steps.
\begin{thm}
Let $K$ be a virtual diagram, whose underlying diagram genus is not
minimal in the class of the knot $K$. Then there exists a diagram
$K'<K$ in the same knot class. \label{lmkey}
\end{thm}

\begin{thm} [I.M.Nikonov]
There is a map $pr$ from minimal genus virtual knot diagrams to
classical knot diagrams such that for every knot $K$ we have
$pr(K)<K$ and if two diagrams $K_{1}$ and $K_{2}$ are related by a
Reidemeister move (performed within the given minimal genus diagram)
then their images $pr(K_{1})$ and $pr(K_{2})$ are related by a
Reidemeister move. \label{lmnik}
\end{thm}

\begin{proof}[Proof of the Main Theorem (Theorem \ref{mainthm})]
We shall construct the projection map in two steps.

 Let $K$ be a virtual knot diagram. If $K$ is of a minimal genus,
then we take ${\bar K}$ to be just $pr(K)$ as in Theorem
\ref{lmnik}. Otherwise take a diagram $K'$ instead of $K$ as in
Theorem \ref{lmkey}. It is of the same knot type as $K$. If the
genus of the resulting diagram is still not minimal, we proceed by
iterating the operation $K'$, until we get to a diagram $K''$ of
minimal genus which represents the class of $K$ and $K''<K$. Now,
set ${\bar K}=pr(K'')$.

One can easily see that if we insert a small classical knot $L$
inside an edge of a diagram of $K$, then $f(K\# L)=f(K)\# f(L)$. So,
the last statement of the theorem holds as well.
\end{proof}

\section{Proofs of Key Theorems}

\subsection{The Proof of Theorem \ref{lmkey}}

Let $K$ be a virtual knot diagram on a surface $S_{g}$ of genus $g$.
Assume this genus is not minimal for the knot class of $K$. Then by
Kuperberg's theorem it follows that there is a diagram ${\tilde K}$
on $S_{g}$ representing the same knot as $K$ and a curve $\gamma$ on
$S_{g}$ such that ${\tilde K}$ does not intersect $\gamma$. Indeed,
if there were no such diagram ${\tilde K}$, the knot in $S_{g}\times
I$ corresponding to the diagram $K$ would admit no destabilization,
and the genus $g$ would be minimal.

The curve $\gamma$ gives rise to a (co)homological parity for knots
in $S_{g}$ homotopic to $K$: a crossing is {\em even} if the number
if intersections of any of the corresponding halves with $\gamma$ is
even, and odd, otherwise.

Since $K$ has underlying diagram genus $g$, there exists at least
one odd crossing of the diagram $K$. Let $K$ be the result of
$\gamma$-parity projection applied to $K$. We have $K'<K$.

By construction, all crossings of ${\tilde K}$ are even.

Let us construct a chain of Reidemeister moves from $K$ to ${\tilde
K}$ and apply the $\gamma$-parity projection to it.

We shall get a chain of Reidemeister moves connecting $K'$ to
${\tilde K}$. So, $K'$ is of the same type as ${\tilde K}$ and $K$.
The claim follows.

\subsection{The Proof of Theorem \ref{lmnik}}
Let us construct the projection announced in Theorem \ref{lmnik}.
Fix a $2$-surface $S_{g}$. Let us consider knots in the thickening
of $S_{g}$ for which genus $g$ is minimal (that is, there is no
representative of lower genus for knots in question). Let $K$ be a
diagram of such a knot. We shall denote crossings of knot diagrams
in $S_{g}$ and the corresponding points on $S_{g}$ itself by the
same letter (abusing notation).

As above, with every crossing $v$ of $K$ we associate the two {\em
halves} $h_{v,1},h_{v,2}$, now considered as elements of the
fundamental group $\pi_{1}(S_{g},v)$, as follows. Let us smooth
 the diagram $K$ at $v$ according to the orientation of $K$. Thus, we
get a two-component oriented link with components $h_{v,1},h_{v,2}$.
Consider every component of this link represented as a loop in
$\pi_{1}(S_{g},v)$ and denote them again by $h_{v,1},h_{v,2}$.

Let $\gamma_{v},{\bar \gamma_{v}}$ be the two homotopy classes of
the knot $K$ considered as an element of $\pi_{1}(S_{g},v)$: we have
two classes because we can start traversing the knot along each of
the two edges emanating from $v$. Note that $h_{v,1}\cdot
h_{v,2}=\gamma_{v}$ and $h_{v,2}\cdot h_{v,1}={\bar \gamma_{v}}$.

Let us now construct a knot diagram $pr(K)$ from $K$ as follows. If
for a crossing $v$ we have $h_{v,1}=\gamma_{v}^{k}$ for some $k$
(or, equivalently, $h_{v,2}=\gamma_{v}^{1-k}$) then this crossing
remains classical for $K'$; otherwise, a crossing becomes virtual.
Note that it is immaterial whether we take $\gamma_{v}$ or ${\bar
\gamma_{v}}$ because if $h_{v,1}$ and $h_{v,2}$ are powers of the
same element of the fundamental groups, then they obviously commute,
which means that $\gamma_{v}={\bar {\gamma_{v}}}$.

\begin{st}
\begin{enumerate}

\item For every $K$ as above, $pr(K)$ is a classical diagram;

\item $K=pr(K)$ whenever $K$ is classical

\item If $K_{1}$ and $K_{2}$ differ by a Reidemeister move then
$pr(K_{1})$ and $pr(K_{2})$ differ by either a detour move or by a
Reidemeister move.

\end{enumerate}
\end{st}

\begin{proof}
Take $K$ as above and consider $pr(K)$. By construction, all
``halves'' of all crossings for $pr(K)$ are powers of the same
homotopy class. We claim that the underlying surface for $pr(K)$ is
a $2$-sphere. Indeed, when constructing a band presentation for
$pr(K)$, we see that the surface with boundary has cyclic homology
group. This happens only for a disc or for the cylinder; in both
cases, the corresponding compact surface will be $S^{2}$.

The situation with the first Reidemeister move is obvious: the new
added crossing has one trivial half and the other half equal to the
homotopy class of the knot itself.

Now, to prove the last statement, we have to look carefully at the
second and the third Reidemeister moves. Namely, if some two
crossings $A$ and $B$ participate in a second Reidemeister move,
then we have an obvious one-to-one correspondence between their
halves such that whenever one half corresponding to $A$ is an power
of $\gamma$, so is the corresponding half of $B$.

So, they either both survive in $pr(A),pr(B)$ (do not become
virtual) or they both turn into virtual crossings. So, for
$pr(A),pr(B)$ we get either the second Reidemeister move, or the
detour move. Note that here we deal with the second Reidemeister
move which does not change the underlying surface.
\begin{figure}
     \begin{center}
     \begin{tabular}{c}
     \includegraphics[width=10cm]{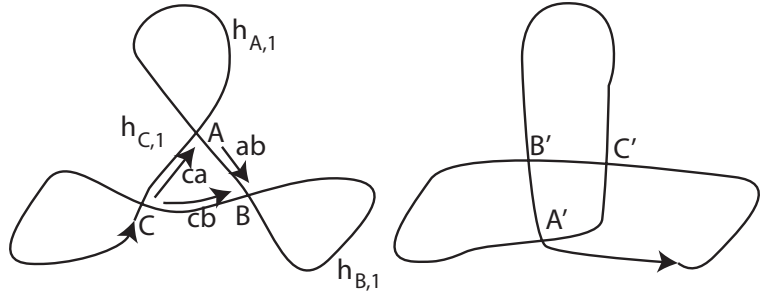}
     \end{tabular}
     \caption{Triviality of two crossings yields the triviality of the third one}
     \label{NikonovFig}
\end{center}
\end{figure}

Now, let us turn to the third Reidemeister move from $K$ to $K'$,
and let $(A,B,C)$ and $(A',B',C')$ be the corresponding triples of
crossings. We  see that the homotopy classes of halves of $A$ are
exactly those of $A'$, the same about $B,B'$ and $C,C'$. So, the
only fact we have to check that the number of surviving crossings
among $A,B,C$ is not equal to two (the crossings from the list
$A',B',C'$ survive accordingly). This follows from Fig.
\ref{NikonovFig}.

Indeed, without loss of generality assume $A$ and $B$ survive. This
means that the class $h_{A,1}$ is a power of the class of the whole
knot in the fundamental group with the reference point in $A$, and
$h_{B,1}$ is a power of class of the knot with the reference point
at $B$.

Let us not investigate $h_{C,1}$ (for convenience we have chosen
$h_{C,1}$ to be the upper right part of the figure).

We see that $h_{C,1}$ consists of the following paths: $(ca)
h_{A,1}(ab)h_{B,1}(cb)^{-1}$, where $(ca), (ab),(cb)$ are non-closed
paths connecting the points $A$, $B$, and $C$. Now, we can homotop
the above loop to $(ca)h_{A,1}(ca)^{-1}(ca)(ab)h_{B,1}(cb)^{-1}$ and
then homotop it to the product of $(ca)h_{A,1}(ca)^{-1}$ and
$(cb)h_{B,1}(cb)^{-1}$.

We claim that both these loops are homotopic to $\gamma_{C}^{l}$ and
$\gamma_{C}^{m}$ for some exponents $m,l$. Indeed, $h_{A,1}$ is
$\gamma_{A}^{k}$ by assumption. Now, it remains to observe that in
order to get from $\gamma_{A}$ to $\gamma_{C}$, it suffices to
``conjugate'' by a path along the knot; one can choose $(ac)$ as
such a path. The same holds about $h_{C,1}$.

So, if all crossings $A,B,C$ survive in the projection of $pr(K)$
and $A',B',C'$ survive in $pr(K')$ then we see that $pr(K')$ differs
from $pr(K)$ by a third Reidemeister move. If no more than one of
$A,B,C$ survives then we have a detour move from $pr(K)$ to
$pr(K')$.

\end{proof}

\section{The Parity Group, One More Projection, and Connected Sums}

In the above text, we have defined parity as a way of decorating
crossings by elements of $\Z_{2}$. It turns out that there is a way
to construct an analogue of parity valued in more complicated
objects, namely, in groups, depending on the knot diagram. Such
``group-valued'' parities can be also used for projections, see,
e.g., \cite{IMN}.

This group-valued parity can be thought of as a parity in a weak
sense: a crossing is even if the corresponding element of the parity
group is trivial, and odd otherwise.

However, this can be done for diagrams of some specific genus only.

Let $D$ be a virtual diagram of genus $g$. Now, let us construct the
{\em universal parity group} $G(D)$. Note that this group will be
``universal'' only for a specific genus.

Recall that pasted cycles appear in a band--pass presentation of a
virtual knot diagram as cycles on the boundary of a surface to be
pasted by discs. Every cycle can be treated as a $1$-cycle in the
$1$-frame of the knot diagram graph; the graph itself consists of
classical crossings (vertices) and edges between them. Thus, every
pasted cycle $C$ gives rise to a collection of classical crossings,
it touches.

We shall use the additive notation for this group. For generators of
$G(D)$ we take crossings of the diagram $D$. We define two sorts of
relations:

\begin{enumerate}

\item
 $2a_{i}=0$ for every crossing and there will also be relations
correspond to {\em pasted cycles}. Namely, a pasted cycle is just a
rotating cycle on the $4$-valent graph (shadow of the knot)

\item The sum
of crossings corresponding to any pasted cycle is zero.

\end{enumerate}

It is obvious that for a classical knot diagram $D$ the group $G(D)$
is trivial (otherwise the reader is referred to Theorem \ref{thth}
ahead).

Denote the element of the group $G$ corresponding to a crossing $x$
of the knot diagram, by $g(x)$.

In \cite{IMN} it is proved that the parity group gives rise to a
parity in a weak sense: all crossings for which the corresponding
element of the group is trivial, are thought of as {\em even}
crossings, and the other one are thought of as {\em odd crossings}.
Thus, we get the following

\begin{thm}
 For a virtual diagram $D$ with the surface $S_{g}$ genus $g$ the
group $G(D)$ is the quotient group of $H_{1}(S_{g},\Z_{2})$ by the
element generated by the knot. In particular, if  $D$ is a
checkerboard colourable diagram then $G(D)=H_{1}(S_{g},\Z_{2})$.

In particular, if $D_{1}$ and $D_{2}$ are nonstably equivalent
diagrams then $G(D_{1})=G(D_{2})$.

\label{thth}

\end{thm}

To prove the theorem, it suffices to associate with every crossing
$x$ any of the two halves $h_{x,1}$ or $h_{x,2}$ and consider them
as elements of the above mentioned quotient group.

A careful look to the formulation of Theorem \ref{thth} shows that:
\begin{enumerate}
\item If a crossing $x$ corresponds to the first Reidemeister move,
then the corresponding element of the quotient group is equal to
zero.

\item If two crossings $x,y$ participate in the second Reidemeister
move, then the corresponding elements of the group $G(D)$ are equal
to each other.

\item If three crossings $a,b,c$ participate in a third Reidemeister
move then $h_{a}+h_{b}+h_{c}=0$ in $G$.

\end{enumerate}

Thus, the map to the group $G$ gives rise to the {\em parity in a
weak sense}, which means, in particular, that there is a
well-defined projection from knots in $S_{g}\times I$ to virtual
knots.

Let $K$ be a knot diagram in $S_{g}$. Consider $K$ as a virtual knot
diagram (up to virtual equivalence). Now, let $l(K)$ be the diagram
obtained from $K$ by making those crossings $x$ of $K$ virtual, for
which $h(X)\neq 0\in G(K)$.

\begin{thm}
If $K$ and $K_{1}$ are two diagrams of knots in $S_{g}\times I$
which differ by one Reidemeister move, then $l(K)$ and $l(K_{1})$
either differ by the same Reidemeister move, or coincide (are
virtually equivalent).

Moreover $l(K)$ is (virtually equivalent to) $K$ if and only if $K$
is (virtually equivalent to) a classical knot.
\end{thm}

The proof follows from general argument concerning parity in a weak
sense.

\subsection{One more projection}

Let us now give one more proof of Theorem \ref{mainthm}. In fact,
the map $f$ from our original proof of Theorem \ref{mainthm} kills
too many classical crossings.

For example, if we consider the classical trefoil diagram with three
``boxes'' shown in Fig.\ref{trefblackbox}. Assume every black box
represents a virtual knot diagram lying inside its minimal
representative which is homologically trivial in the corresponding
$2$-surface, and we put these diagrams into boxes after splitting
them at some points. Then, if these diagrams are complicated enough
then we see that all three middle classical crossings will become
virtual after applying Nikonov's projection.

\begin{figure}
\centering\includegraphics[width=200pt]{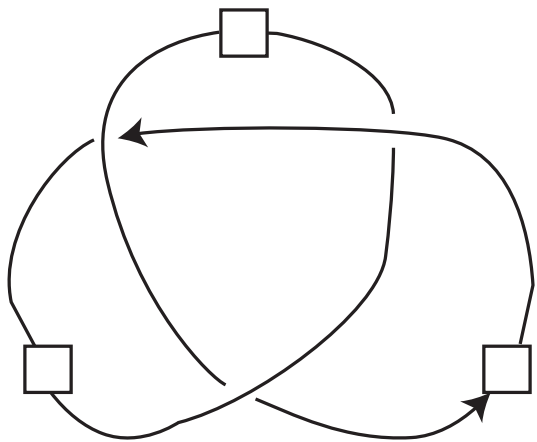}
\caption{Classical trefoil with black boxes}
 \label{trefblackbox}
\end{figure}

On the other hand, since these three virtual knots are homologically
trivial, their persistence does not affect the homological
triviality of the three crossings depicted in Fig.
\ref{trefblackbox}. So, there is a motivation how to find another
projection satisfying the condition of Theorem \ref{mainthm} which
does not kill the three crossings depicted in this Figure.

The reason is that the Nikonov projection is very restrictive and
makes many classical crossings virtual.

Let us now construct another map $g$ from virtual knots to classical
knots satisfying all conditions of Theorem \ref{mainthm}.

Take a virtual knot diagram $K$. If it is not a minimal genus
diagram, apply Theorem \ref{lmkey}. We get a diagram $K'$. If $K'$
is not yet of the minimal genus, apply Theorem \ref{lmkey} until we
get to a mininal genus diagram. Take this minimal genus diagram
$K_{m}$ and apply the projection with respect to the parity group.
Then (if necessary) we again reiterate Theorem \ref{lmkey} to get to
the minimal genus diagram, and then apply the parity projection
once.

Every time we shall have a mapping which is well defined on the
classes of knots: Theorem \ref{lmkey} does not change the class of
the knot at all, and the group parity projection is well defined
once we know that we are on the minimal genus.

The resulting diagram will be classical. Denote it by $g(K)$.

The reader can easily find virtual knots (1-1 tangles) to be
inserted in Fig. \ref{trefblackbox}, so that for the resulting knot
$K$, the projection $g(K)$ gives the trefoil knot, whence the
projection $f(K)$ is the unknot.

For exact definitions of connected sums, see \cite{MyNewBook,KM}
\begin{cj}
The map $g$ takes connected sum of virtual knots to connected sums
of classical knots.
\end{cj}

Of course, there are ways to mix the approaches described in the
present paper to construct further projections satisfying the
conditions of Theorem \ref{mainthm}.

An interesting question is to find ``the most careful'' projection
satisfying all conditions of Theorem \ref{mainthm} which preserves
more classical data.

\section{Problems with the existence of a well defined map on
diagrams}

Consider the virtual knot diagram $A$ drawn in the left picture of
Fig. \ref{vkclasproj}. If we seek a projection satisfying conditions
of the Main Theorem, we may $A$ project to $D$ in the same picture
(lower right). Note that $A$ is not classical. However, the two
intermediate knots ($B$ and $C$) are both classical: they are drawn
on the torus, however, they both fit into a cylinder, and hence, to
the plane; so, they will project to themselves.

There is no obvious reason why the projection of $A$ should be
exactly $D$ because both $B$ and $C$ are classical; on the other
hand there is no obvious way to make a preferred choice between $B$
and $C$ if one decides to take them to be the result of projection
of $A$.

So, a bigger diagram projects to a smaller one (we see that
$A>B,A>C$ but $B>D,C>D$). This is the lack of naturality which does
not allow one to make projection compatible with Reidemeister moves.
Of course $A$ differs from $B$ by one Reidemeister move, as well as
their images $D$ and $B$, but in the first case the move is
decreasing, and in the second case it is increasing.

This is also the reason of ambiguity: in fact, one can also project
$A$ to $B$ or $C$ since both these diagrams are classical.

\begin{figure}[htb]
     \begin{center}
     \begin{tabular}{c}
     \includegraphics[width=10cm]{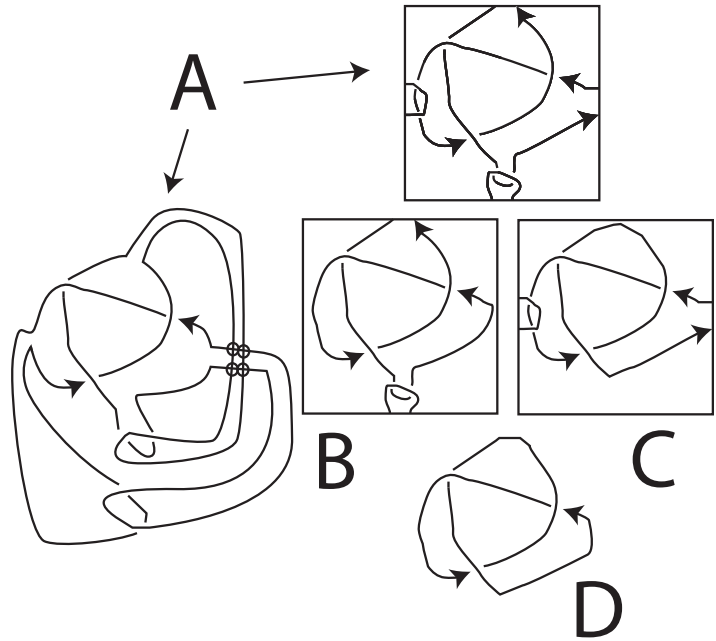}
     \end{tabular}
     \caption{Virtual Knot and Its Classical Projection}
     \label{vkclasproj}
\end{center}
\end{figure}

\end{document}